\documentclass[10pt, reqno]{amsart}
\usepackage{amsmath,amsfonts,amsthm,bm} 

\usepackage{amsfonts}
\usepackage{amssymb}
\usepackage{latexsym}

\usepackage{tikz}
\usepackage{graphicx,graphics,epsfig}
\usepackage{epstopdf}	
\usepackage[noadjust]{cite}

\DeclareMathOperator{\sech}{sech}

\usepackage[]{enumerate}
\usepackage{verbatim}
\usepackage{bbm}

\usepackage[]{enumerate}

\usepackage[sc]{mathpazo}          
\usepackage{eulervm}               
\usepackage[scaled=0.86]{berasans} 
\usepackage[scaled=1]{inconsolata} 
\usepackage[T1]{fontenc}

\usepackage[final]{hyperref}
\hypersetup{colorlinks=true, linkcolor=blue, anchorcolor=blue, citecolor=red, filecolor=blue, menucolor=blue, pagecolor=blue, urlcolor=blue}

\newcommand{\Bigabs}[1]{\Bigl\vert #1 \Bigr\vert}

\newcommand{\norm}[1]{\left\Vert #1 \right\Vert}

\newcommand{\R}{\mathbb{R}}

\newcommand{\angles}[1]{\langle #1 \rangle}

\newtheorem{theorem}{Theorem}

\newtheorem{lemma}{Lemma}

\newtheorem*{PWtheorem}{Paley-Wiener Theorem}

\theoremstyle{definition}

\theoremstyle{remark}
\newtheorem{remark}{Remark}

\title[Spatial analyticity for the Beam Equation]{On the persistence of spatial  analyticity for the Beam Equation}

\author[T. T. Dufera] {Tamirat T. Dufera}

  \author [
S. Mebrate] { 
Sileshi Mebrate }

  \author[A. Tesfahun]{Achenef Tesfahun}

\address{Department of Mathematics \\
Nazarbayev University \\
Qabanbai Batyr Avenue 53 \\
010000 Nur-Sultan \\
Republic of Kazakhstan}

\email{achenef@gmail.com}

\address{Department of Mathematics
\\
Adama University of Science and Technology
\\
Ethiopia}
 
   \email{tamirat.temesgen@astu.edu.et, silenaty2005@gmail.com}


\begin{document}

    \begin{abstract}
    Persistence of spatial analyticity is studied for solution of the beam equation $ u_{tt} + \left(m+\Delta^2\right) u + |u|^{p-1}u = 0$ on $\R^n \times \R$. In particular, 
for a class of analytic initial data with a uniform radius of analyticity $\sigma_0$, we obtain an asymptotic lower bound 
 $\sigma(t) \ge c/\sqrt t$ on the uniform radius of analyticity $\sigma(t)$ of solution $u(\cdot, t)$, as $t \rightarrow \infty.$
\end{abstract}

\maketitle

\section{Introduction}
This paper is concerned with persistence of spatial analyticity of solutions for
the Cauchy problem of the fourth order wave equation
\begin{equation}\label{NLB}
\left\{
\begin{aligned}
    & u_{tt} + \left(m+\Delta^2\right) u + |u|^{p-1}u = 0,
    \\
    & (u, u_t )|_{t=0}= (u_0 , u_1),
\end{aligned}
\right.
\end{equation}
where $u: \R^{n} \times \R  \longrightarrow \mathbb{R}$, $p \ge
 1$, and $m > 0$. 

 Equations like \eqref{NLB}
are also referred to as Bretherton’s type equations or the beam equation. The original Bretherton equation, written down for $n = 1$ by Bretherton \cite{Breth1964}, arised in the study of weak interactions of dispersive waves. A similar equation for $n = 2$ was proposed in Love \cite{Love1944} for the motion of a clamped plate. Recent developments in arbitrary dimension were established in \cite{L98, L98e, BS07, BS00, LS00}.

Questions such as well-posedness, blow-up in finite time, long time existence, and the existence of uniform bounds for global solutions of \eqref{NLB} are adressed by several authors.
For instance, local well-posedness, scattering, and stability in the energy space $H^{2} (\R^n) \times L^{2} (\R^n) $ was studied by Levandosky in \cite{L98, L98e} and Levandosky and Strauss \cite{LS00}, results which were extended by Pausader \cite{P07,P10}.  Low-regularity global well-posedness was also shown by Zhang \cite{Z10} in dimensions $3 \leq n \leq 7$ in the cubic case for data $ (u_0 , u_1) \in H^{s} (\R^n)  \times H^{s-2} (\R^n) $ satisfying
\[
s > \min \left\{ \frac{n-2}{2},\ \frac{n}{4} \right\}.
\]
We remark that the energy 
\[
E(t)=\frac12 \int_{\R^n} \left( u^{2}_{t}+  (\Delta u)^2  + m u^{2}+  \frac2{p+1}  |u|^{p+1} \right) \ dx
\]
is conserved by the flow of \eqref{NLB}, i.e., $E(t)=\text{const.}$ for all $t$.

In this paper, we shall study the persistence of spatial analyticity for the solution of the Cauchy problem \eqref{NLB}, given initial data in a class of analytic functions. 
By the Paley-Wiener Theorem, the radius of analyticity of a function can be related to decay properties of its Fourier transform. It is therefore natural to take data for \eqref{NLB} in the Gevrey space $G^{\sigma, s}(\R^n)$,  defined by the norm
\[
\| f \|_{G^{\sigma, s}(\R^{n})} = \| \exp(\sigma |\xi|) \langle \xi \rangle^{s} \hat{f} \|_{L^2_\xi (\R^n)}\qquad (\sigma\ge 0),
\]
where $\angles{\xi }= \sqrt{1+ |\xi|^2}$. When 
$\sigma= 0$, this space coincides with the Sobolev space $H^{s}(\R^n)$, with norm
\[
\| f \|_{H^{s}(\R^{n})} = \| \langle \xi \rangle^{s} \hat{f} \|_{L^2_\xi (\R^n)},
\]
while for any $\sigma>0$, any function in
  $G^{\sigma, s}(\R^n)$ has a radius of analyticity of at least $\sigma$ at each point $x\in \R^n$.
This fact is contained in the following theorem, whose proof can be found in \cite{K1976} in the case $s = 0$ and $n = 1$; the general case follows from a simple modification.  

\begin{PWtheorem}
Let $\sigma > 0$ and $s \in \mathbb{R}$.  If $f \in G^{\sigma, s}(\R^{n})$, then $f$ is the restriction to $\R^n$ of a function $F$ which is holomorphic in the strip
\[
    S_{\sigma} = \{ x + iy \in \mathbb{C}^{n}:\ |y| < \sigma \}.
\]
Moreover, the function $F$ satisfies the estimates
\[
    \sup_{|y| < \sigma} \| F (\cdot + i y) \|_{H^{s}} < \infty.
\]
\end{PWtheorem}

Information about the domain of analyticity of a solution
to a partial differential equation (PDE)
can be used to gain a quantitative understanding
of the structure of the equation,
and to obtain insight into underlying physical processes. 
The study of real-analytic solutions to nonlinear PDE
has developed over the last three decades.
Starting with the works of Kato and Masuda \cite{KM1986}
for dispersive wave-type equations, and Foias and Temam \cite{FT1989}
for the Navier-Stokes equations, analytic function
spaces have become popular tools for the study
of a variety of questions connected with nonlinear evolutionary PDE.
In particular, the use of Gevrey-type spaces has given
rise to a number of important results
in the study of long time dynamics of dissipative
equation, such as
estimating the asymptotic degrees of freedom (e.g.,
determining nodes) \cite{OT}, 
approximating the global attractors/inertial
manifolds \cite{JT} and and a rigorous estimate 
of the Reynold's scale \cite{CoDoTi}.

Given a nonlinear dispersive PDE in the independent variables $(t,x)$, consider the Cauchy problem with real analytic initial data at $t=0$. If these data have a uniform radius of analyticity $\sigma_0$, in the sense that there exists a holomorphic extension to the complex strip of width $\sigma_0$, then we ask whether the solution at some later time $t>0$ also has a uniform radius of analyticity $\sigma=\sigma(t)>0$, in which case we would, moreover, like to have an explicit lower bound on $\sigma(t)$. 
Heuristically, the picture one should have in mind is that $\sigma(t)$ is the distance from the $x$-axis to the nearest complex singularity of the holomorphic extension of the solution at time $t$. If at some time $t$ this singularity actually hits the x-axis, then the solution itself suffers a breakdown of regularity. This point of view is the basis for the widely used singularity tracking method \cite{SSF83} in numerical analysis, where a spectral method is used to obtain a numerical estimate of $\sigma(t)$. This estimate can then be used to predict either the formation of a singularity in finite time or alternatively global regularity.  Even in cases where singularity formation does not occur (as is the case for the beam equation), it is still of interest to obtain lower bounds on $\sigma(t)$, as this has implications for the rate of convergence of spectral methods for the equation one is looking at (see \cite{BH07} for an example of this).

The spaces $G^{\sigma, s}(\R^n)$ were introduced by Foias and Temam \cite{FT1989} (see also \cite{KM1986}) in the study of spatial analyticity of solutions to the Navier-Stokes equations, and various refinements of their method have since been applied to prove lower bounds on the radius of spatial analyticity for a number of nonlinear evolution equations \cite{Ferrari1998, Oliver2001, Panizzi2012, HHP2011, HKS2017, HP2012,  Levermore1997, ST2015, ST2017, SD2015, AT2017, AT2019, AT2019-e}. The method used here for proving lower bounds on the radius of analyticity was introduced in \cite{ST2015} in the context of the 1D Dirac-Klein-Gordon equations. This method is based on an approximate conservation laws, and has been applied to prove an
algebraic lower bound (decay rate) of order $t^{-1/\alpha}$ for some $\alpha \in (0, 1]$ on the radius of spatial analyticity of solutions to a number of nonlinear dispersive and wave equations (see e.g.,  \cite{BTT21, ST2015, ST2017, SD2015, AT2017, AT2019, AT2019-e}). The optimal decay rate that can be obtained in this setting is $1/t$, which corresponds to $\alpha=1$ (see e.g.,  \cite{AT2017, AT2019-e, BTT21}). 
 This decay rate is related to the behavior of the exponential weight, $\exp(\sigma |\xi|) $, that sits in the Gevrey norm. More specifically, it stems from the simple estimate
$$\exp(\sigma |\xi|) -1\le ( \sigma |\xi|)^\alpha  \cdot  \exp(\sigma |\xi|)  \qquad (0\le \alpha \le 1),$$
which follows from an interpolation between $\exp r-1\le \exp r $ and $ \exp r-1\le r  \exp r$
for $r\ge 0$.

In the present work, in an attempt to improve  the decay rate obtained so far,
we introduce a modified Gevrey norm\footnote{As far as we now the space $H^{\sigma, s}$ is new to this paper and was not used before. } by
\[
\| f \|_{H^{\sigma, s}(\R^{n})} = \|  \cosh ( \sigma |\xi|) \langle \xi \rangle^{s} \hat{f} \|_{L^2_\xi (\R^n)} \qquad (\sigma\ge 0),
\]
where the exponential weight $\exp(\sigma |\xi|)$ in the Gevrey norm is now replaced by a hyperbolic weight $\cosh(\sigma |\xi|) $. These two weights are equivalent in the sense that
\begin{equation}
\label{cosh:est}
\frac 12 \exp(\sigma |\xi|)    \le \cosh (\sigma |\xi|) \le \exp(\sigma |\xi|).
\end{equation}
Thus, the associated $G^{\sigma, s}$ and $H^{\sigma, s}$--norms are equivalent, i.e.,
\begin{equation}\label{GH-eqiv}
\| f \|_{H^{\sigma, s} (\R^n)} \sim \| f \|_{G^{\sigma, s} (\R^n)},
\end{equation}
and so the statement of Paley-Wiener Theorem still holds for functions in $H^{\sigma, s}$.

Note also that
$H^{0, s}=G^{0, s}=H^{ s}$.
The space $H^{\sigma, s}$, however, has an advantage since $\cosh(\sigma |\xi|) $ satisfies the estimate
\begin{equation}
\label{cosh-1:est}
\cosh ( \sigma |\xi|)-1\le  (\sigma |\xi|)^{2\alpha}  \cdot \cosh ( \sigma |\xi|) \qquad (0\le \alpha \le 1).
\end{equation}
This follows from 
$$\cosh r-1\le \cosh r \quad \text{ and } \quad \cosh r-1\le r^2  \cosh r   \qquad (r\in \R).$$ 
Therefore, in view of \eqref{cosh-1:est}, an application of  our method in the $H^{\sigma, s}$-set up can yield a decay rate of order $t^{-1/2\alpha}$ for some $\alpha \in (0, 1]$ provided that the nonlinear estimates in the derivation of the approximate conservation law can absorb the weight $|\xi|^{2\alpha}$. In this work
we managed to obtain the optimal decay rate of $t^{-1/2}$ (which corresponds to $\alpha=1$) for the Cauchy problem \eqref{NLB}.

We remark that as  a consequence 
 the embedding
\begin{equation}\label{embed1}
H^{\sigma, s}  \subset  H^{s}    \qquad  \   (  \sigma\ge 0)
\end{equation}
 and the existing well-posedness theory in $H^2(\R^n) \times L^{2}(\R^n)$, one can conclude that the Cauchy problem \eqref{NLB}, with $1\le n\le 3$ and $p\ge 1$, has a unique, global-in-time solution, given initial data 
$(u_0, u_1) \in H^{\sigma_0, 2}(\R^n)\times H^{\sigma_0, 0}(\R^n) $ for some $\sigma_0\ge 0$.

\vspace{2mm}

We now state our main theorem.
\begin{theorem}[Lower bound for the radius of analyticity]\label{thm-gwp}
Let $1\le n \le 3$,  $p\ge 1$ is an odd integer and $\sigma_0 > 0$. If
 $(u_0, u_1) \in H^{\sigma_0, 2} (\R^n) \times H^{\sigma_0, 0} (\R^n)$, then for any $T > 0$ the solution of \eqref{NLB} satisfies
\[
 (u, u_t) \in C\left( [0, T]; H^{\sigma, 2}(\R^n) \right) \times C^1\left( [0, T]; H^{\sigma, 0}(\R^n) \right) 
\]
with
\[
\sigma:=\sigma(T) = \min \left\{ \sigma_0,   c T^{-\frac12}\right\}  ,
\]
where $c > 0$ is a constant depending on 
 the initial data norm.
\end{theorem}
\noindent In view of the Paley-Wiener theorem and \eqref{GH-eqiv}, this result implies that the solution $u(\cdot, t)$ has radius of analyticity at least $\sigma(t)$ for every $t>0$.

\begin{remark} The result of Theorem \ref{thm-gwp} can be extended to dimension $n=4$ if one uses Strichartz estimates (see \cite{BS07} ) instead of just Sobolev embeddings as we do in this paper. It is also possible to extend the result for $n\ge 5$ but with some upper bound restriction on $p$.
However, we will not pursue these issues here.
\end{remark}

\vspace{2mm}
The first step is to prove the following local-in-time result, where the radius of analyticity remains constant.

\begin{theorem}[Local well-posedness]\label{thm-lwp}
Let $1\le n \le 3$,  $p\ge 1$ is an odd integer and $\sigma > 0$. Given $(u_0, u_1) \in H^{\sigma, 2}(\R^n)\times H^{\sigma, 0} (\R^n)$, there exists a time $\delta>0$ and a unique solution
\[
 (u, u_t) \in C\left( [0, \delta]; H^{\sigma, 2} (\R^n)\right) \times C^1\left( [0, \delta]; H^{\sigma, 0}(\R^n) \right) 
\]
of the Cauchy problem \eqref{NLB} on $\R^n \times [0, \delta]$. Moreover,  
  the existence time is given by 
\begin{equation}\label{delta}
 \delta=c_0 \left(  
  \norm{ u_0 }_{ H^{ \sigma, 2} }+ \norm{ u_1 }_{ H^{ \sigma, 0} }\right)^{-(p-1)}.
 \end{equation}
for some constant $c_0>0$.
\end{theorem}

\vspace{2mm}

The second step in the proof of Theorem \ref{thm-gwp} is to prove an approximate conservation law for the norm of the solution, that involves a small parameter $\sigma >0$ and which reduces to the exact energy conservation law in the limit as $\sigma \rightarrow 0$. 
To derive this approximate conservation law, we  set
\[
v_\sigma(x,t): =\cosh (\sigma |D|)  u(x,t),\]
where $u$ is the solution to \eqref{NLB}. 
Define a modified energy associated with $v_\sigma$ by
\[
\mathcal E_{\sigma}(t)=\frac12 \int_{\R^n} \left( (\partial_t v_\sigma)^{2}+  (\Delta v_\sigma)^2  + m v_\sigma^{2}+  \frac2{p+1}  |v_\sigma|^{p+1} \right) \ dx.
\]

\begin{theorem}[Approximate conservation law]\label{thm-almostconsv}
Let $1\le n \le 3$,  $p\ge 1$ is an odd integer and  $\sigma > 0$. Given $(u_0, u_1) \in H^{\sigma, 2}\times H^{\sigma, 0} $, let $u$ be the local solution of \eqref{NLB} on $\R^n \times [0, \delta]$ that is obtained in Theorem \ref{thm-lwp}. Then
\begin{equation}\label{approx-Energy}
\sup_{0 \le t\le \delta}  \mathcal E_\sigma(t) =   \mathcal E_\sigma(0) +  \delta \sigma^2   \cdot  \mathcal O \left(    \left( \mathcal E_{\sigma}(0)\right)^{\frac{p+1}2}   \right).
\end{equation} 
\end{theorem}
\noindent Observe that in the limit as $\sigma \rightarrow 0$, we recover the conservation $ \mathcal E_0(t)=  \mathcal E_0(0) $ for $0\le t \le \delta$ (note that $v_0=u$).
Applying the last two theorems repeatedly, and then by taking $\sigma>0$ small enough we can cover any time interval $[0, T]$ and obtain the main result, Theorem \ref{thm-gwp}.

\vspace{2mm}

\medskip
\noindent \textbf{Notation}. 
For any positive numbers $a$ and $b$, the notation $a\lesssim b$ stands for $a\le cb$, where $c$ is a positive constant that may change from line to line. Moreover, we denote $a \sim b$  when  $a \lesssim b$ and $b \lesssim a$.

\vspace{2mm}

In the next sections we prove Theorems \ref{thm-lwp}, \ref{thm-almostconsv} and \ref{thm-gwp}.

\section{Proof of Theorem \ref{thm-lwp} }

Theorem \ref{thm-lwp} can easily be proved using energy inequality, Sobolev embedding and a standard contraction argument. Indeed, consider 
the Cauchy problem for the linear beam equation
\begin{equation*}
\begin{split}
    & u_{tt} + \left(m+\Delta^2\right) u = F,
    \\
    & (u, u_t )|_{t=0}= (u_0 , u_1)
    \end{split}
\end{equation*}
whose solution is given by Duhamel's formula
\begin{equation}
\label{duml}
u(t) = \mathcal S'_m(t) u_0 + \mathcal S_m(t) u_1 + \int_{0}^{t}  \mathcal S_m(t-s) F( s)\ ds,
\end{equation}
where
$$
\mathcal S_m(t)= \frac{ \sin\left( t \angles{\Delta}_m\right)}{ \angles{\Delta}_m}.
$$
Applying $\cosh(\sigma |D|)$ to \eqref{duml} and taking the $H^2$-norm on both sides
yields
the energy inequality 
\begin{equation}
\label{energyinq}
\sup_{0 \le t \le \delta}  \left( \norm{u}_{H^{\sigma, 2}} + \norm{u_t}_{H^{\sigma, 0}}   \right) \lesssim  \norm{u_0}_{H^{\sigma, 2}} + \norm{u_1}_{H^{\sigma, 0}} +  \int_{0}^{\delta}  \norm{F(s)}_{ H^{\sigma, 0}} \ ds
\end{equation}
for some $\delta >0$.

Now consider the integral formulation of \eqref{NLB},
\begin{equation}
\label{lntF}
u(t) = \mathcal S'_m(t) u_0 + \mathcal S_m(t) u_1 + \int_{0}^{t}  \mathcal S_m(t-s) u^p ( s)\ ds,
\end{equation}
where we used the fact that $ |u|^{p-1}u=u^p$ for odd $p$.

Then by \eqref{energyinq} and a standard contraction argument, Theorem \ref{thm-lwp} reduces to proving the nonlinear estimate
\begin{equation}
\label{NonlinearEst}
\norm{u^p}_{ H^{\sigma, 0}}  \lesssim \norm{u}^p_{ H^{\sigma, 2}} ,
\end{equation}
which is also equivalent, by \eqref{GH-eqiv}, to proving  
\begin{equation*}
\norm{u^p}_{ G^{\sigma, 0}}  \lesssim \norm{u}^p_{ G^{\sigma, 2}} .
\end{equation*}
Setting $U=\exp(\sigma |D|)  u$, this estimate reduces further to 
\begin{equation}
\label{NonlinearEst1}
\norm{ \exp(\sigma |D|)  \left[    \left( \exp(-\sigma |D|) U \right)^p   \right] }_{ L^2_x}  \lesssim \norm{U}^p_{ H^{2}} .
\end{equation}
By Plancherel,
\begin{equation*}
 \begin{split}
\text{LHS } \ \eqref{NonlinearEst1}&=
\left\|\mathcal F_x \left\{ \exp(\sigma |D|) \left[    \left( \exp(-\sigma |D|) U \right)^p   \right] \right\}(\xi) \right\|_{L^2_\xi}
\\
& =\left\| \int_{\xi= \sum_{j=1}^p \xi_j} 
   \exp \left( \sigma  \left[  |\xi|-\sum_{j=1}^p |\xi_j| \right]    \right)  \prod_{j=1}^p
 \hat{U}(\xi_{j})  \ d\xi_1 d\xi_2 \cdots d\xi_p\right\|_{L^2_\xi}
 \\
& \le \left\| \int_{\xi= \sum_{j=1}^p \xi_j} 
\prod_{j=1}^p
 |\hat{U}(\xi_{j})|  \ d\xi_1 d\xi_2 \cdots d\xi_p \right\|_{L^2_\xi}
  \\
& =\norm{V^p}_{ L^2_x} 
\end{split}
\end{equation*}
where $V= \mathcal F^{-1} \left[ |\widehat{U}| \right]$. To obtain the third line we used the fact that $ |\xi|\le\sum_{j=1}^p |\xi_j| $, which follows from the triangle inequality.

Now by Sobolev embedding,
\begin{align*}
\norm{V^p}_{ L^2_x} = \norm{V}^p_{ L^{2p}_x} \lesssim  \norm{V}^p_{ H^2} =\norm{U}^p_{ H^2} 
\end{align*}
for all \footnote{This estimate also holds for $1\le p \le \frac{n}{n-4}$ if $n\ge 5$. } $p\ge 1$ if $1\le n \le 4$. This concludes the proof of \eqref{NonlinearEst1}, and hence \eqref{NonlinearEst}.

\section{Proof of Theorem \ref{thm-almostconsv} and Theorem \ref{thm-gwp} } 

Let $1\le n\le 3$, $p\ge 1$ is an odd integer, $\sigma >0$, and 
 $\delta>0$ be the local existence time for the solution obtained in Theorem \ref{thm-lwp}, \eqref{delta}. Recall that $v_\sigma(x,t) =\cosh (\sigma |D|)  u(x,t)$, where $u$ is the solution to \eqref{NLB}. Thus, $u(x,t) =\sech (\sigma |D|)  v_\sigma(x,t)$.

\subsection{Proof of Theorem \ref{thm-almostconsv}}

Using integration by parts and equation \eqref{NLB}, we obtain
\begin{align*}
\mathcal E'_{\sigma}(t) &= \int_{\R^n}  \partial_tv_\sigma \left[ \partial_{tt} v_\sigma +\Delta^2 v_\sigma  +mv   +   v_\sigma^{p} \right]   \ dx
\\
&= \int_{\R^n}   \partial_tv_\sigma  \left[ \cosh (\sigma |D|)  \left(  u_{tt}+  \Delta^2 u   +mu \right) +    v_\sigma^{p} \right ]  \ dx
\\
&= \int_{\R^n}   \partial_tv_\sigma  \left[ -\cosh (\sigma |D|)    u^p  +    v_\sigma^{p} \right ]  \ dx
\\
&= \int_{\R^n}  \partial_t v_\sigma \cdot  N_p\left(  v_\sigma \right)  \ dx,
\end{align*}
where
\begin{equation*}
N_p( v_\sigma)=   v_\sigma^p - \cosh (\sigma |D|) \left[  \sech (\sigma |D| )  v_\sigma  \right]^p.
\end{equation*}
Therefore,
 \begin{equation}\label{Energy-mod1} 
 \begin{split}
\mathcal E_{\sigma}(t)
&= \mathcal E_{\sigma} (0) +\int_{0}^{t} \int_{\R^n} \partial_t v_\sigma  (x,s)  \cdot N_p\left(  v_\sigma (x, s) \right) \, dx ds.
\end{split}
 \end{equation}

\vspace{2mm}
We now state a key estimate that will be proved in the next section.
\begin{lemma}\label{lm-error-est}
For $\partial_t v_\sigma \in L_x^2 $ and $ v_\sigma \in H^2 $,
we have the estimate
\begin{equation}\label{ErrorEst}
 \Bigabs{\int_{\R^n} \partial_t v_\sigma  \cdot  N_P\left( v_\sigma \right) \, dx }     \le C \sigma^2  \norm{
\partial_t v_\sigma  }_{ L_x^2}  \norm{
 v_\sigma  }^p_{ H^2}
\end{equation}
for some constant $C>0$.
\end{lemma}

Now we use \eqref{Energy-mod1} and \eqref{ErrorEst} to obtain the a priori energy estimate
\begin{equation}\label{approx1}
\sup_{0 \le t\le \delta} \mathcal E_{\sigma} (t) 
=  \mathcal E_{\sigma}(0) + \delta  \sigma^2   \cdot \mathcal O \left( \norm{
\partial_t v_\sigma  }_{ L_\delta^\infty L_x^2} \norm{
 v_\sigma  }^p_{ L_\delta^\infty H^2} \right),
\end{equation} 
where we use the notation
$$
 L_\delta^\infty X: =  L_t^\infty X\left([0, \delta]\times \R^n \right)
$$
with $X= L_x^2$ or $H^2$.

As a consequence of  Theorem \ref{thm-lwp} we get 
\begin{equation}\label{v-databd}
\begin{split}
 \norm{  v_\sigma  }_{ L_\delta^\infty H^2}
 + \norm{\partial_t v_\sigma }_{ L_\delta^\infty L_x^2}&=  \norm{u }_{L_\delta^\infty H^{\sigma, 2}} +  \norm{u_t}_{L_\delta^\infty H^{\sigma, 0}} 
 \\
&\le C \left( \norm{u_0 }_{  H^{\sigma, 2}} +  \norm{u_1 }_{  H^{\sigma, 0}} \right)
\\
 &=   C \left( \norm{v_\sigma(\cdot, 0)}_{  H^{ 2}} +  \norm{\partial_t v_\sigma(\cdot, 0)  }_{  L_x^{ 2}} \right).
 \end{split}
\end{equation}
Since 
\begin{equation*}
\begin{split}
\mathcal E_{\sigma}(0) &= \frac12 \int_{\R^n} \left( [\partial_t v_\sigma (x, 0)]^{2}  +  (\Delta  v_\sigma(x, 0) )^2  + m [v_\sigma(x, 0) ]^{2}+  \frac2{p+1}  |v_\sigma(x, 0) |^{p+1} \right) \ dx
\\
& \gtrsim \left( \norm{ v_\sigma(\cdot, 0)  }_{  H^{ 2}} +  \norm{ \partial_t v_\sigma(\cdot, 0)  }_{  L_x^{ 2}} \right)^2
\end{split}
\end{equation*}
it follows from \eqref{v-databd} that
\begin{equation}\label{v-databd1}
\begin{split}
 \norm{  v_\sigma  }_{ L_\delta^\infty H^2}
 + \norm{\partial_t v_\sigma }_{ L_\delta^\infty L_x^2} \lesssim  [ \mathcal E _{\sigma}(0)]^\frac12.
 \end{split}
\end{equation}
Finally, using \eqref{v-databd1} in \eqref{approx1} we obtain the desired estimate \eqref{approx-Energy}.

\subsection{Proof of Theorem \ref{thm-gwp}}

Suppose that $(u_0, u_1) \in H^{\sigma_0, 2}\times H^{\sigma_0, 0} $ for some $\sigma_0 > 0$. This implies 
$$
v_{\sigma_0} (\cdot, 0)=\cosh (\sigma_0 |D)|) u_0 \in H^2, \qquad \partial_t v_{\sigma_0} (\cdot, 0)= \cosh (\sigma_0 |D)|) u_1 \in L^2.
$$
Then by Sobolev embedding
\begin{align*}
\mathcal E_{\sigma_0}(0) \lesssim  \norm{v_{\sigma_0} (\cdot, 0)}^2_{H^2}  +  \norm{\partial_t v_{\sigma_0} (\cdot, 0)}^2_{L_x^2} +  \norm{v_{\sigma_0} (\cdot, 0)}^{p+1}_{H^2}   <\infty.
\end{align*}
Now following the argument in \cite{ST2015} (see also \cite{SD2015})
we can construct a solution on $[0, T]$ for arbitrarily large time $T$ by applying the approximate conservation \eqref{approx-Energy}, so as to repeat the local result in Theorem \ref{thm-almostconsv} on successive short time intervals of size $\delta $ to reach $T$ by adjusting the strip width parameter $ \sigma \in (0, \sigma_0]$ of the solution according to the size of $T$. 

The goal is to prove that for a given parameter $\sigma \in (0, \sigma_0]$  and large $T>0$,
\begin{equation}\label{keybound}
\sup_{t\in [0, T]}  \mathcal E_\sigma(t) \le 2 \mathcal E_{\sigma_0}(0) \qquad \text{for} \quad \sigma = c/\sqrt T ,
\end{equation}
where $c > 0$ depends only on the initial data norm, $\sigma_0$ and $p$. This would imply $\mathcal E_\sigma(t) < \infty$ for all $t \in [0,T]$, and hence
\[
 (u, u_t) (\cdot, t) \in H^{\sigma, 2} \times H^{\sigma, 0} \quad \text{for } \quad \sigma = c/\sqrt T  \quad \text{and}  \ t\in [0, T].
\]

\vspace{2mm}

To prove \eqref{keybound} first observe that for a given parameter $\sigma \in (0, \sigma_0]$
and $0< t_0\le \delta$ we have by
Theorems \ref{thm-lwp} and 
\ref{thm-almostconsv},
\begin{align*}
\sup_{t\in [0, t_0]}  \mathcal E_\sigma(t) \le  \mathcal E_{\sigma}(0) + C\delta \sigma^2  \left(  \mathcal E_{\sigma}(0)\right)^{\frac{p+1}2} \le  \mathcal E_{\sigma_0}(0) + C\delta \sigma^2  \left(  \mathcal E_{\sigma_0}(0)\right)^{\frac{p+1}2} ,
\end{align*}
where we also used the fact the $ \mathcal E_{\sigma}(0) \le \mathcal E_{\sigma_0}(0)$ for $\sigma \le \sigma_0$; this holds since $\cosh r$ is increasing for $r\ge 0$.
Thus,
\begin{equation}\label{Et0}
\sup_{t\in [0, t_0]}  \mathcal E_\sigma(t) \le 2 \mathcal E_{\sigma_0}(0).
\end{equation}
provided
\begin{equation}\label{sigmabd-1}
C\delta \sigma^2  \left(  \mathcal E_{\sigma_0}(0)\right)^{\frac{p+1}2} \le  \mathcal E_{\sigma_0}(0).
\end{equation}

Then we can apply Theorem \ref{thm-lwp}, with initial time $t=t_0$ and the time step $\delta$ as in  \eqref{delta}
to extend the solution to $[t_0, t_0+ \delta]$.
By Theorem \ref{thm-almostconsv}, the approximate conservation law,  and \eqref{Et0} we have
\begin{equation}
  \label{ind1}
  \sup_{[t_0,  t_0+\delta]} \mathcal E_\sigma(t) \le \mathcal E_\sigma(t_0) + C\delta  \sigma^2   \left( 2 \mathcal E_{\sigma_0}(0)\right)^{\frac{p+1}2} .
  \end{equation}
In this way, we cover time intervals $[0, \delta]$, $[\delta, 2\delta]$ etc., and obtain
\begin{align*}
\mathcal E_\sigma(\delta) & \le \mathcal E_\sigma(0) + C \delta  \sigma^2   \left( 2 \mathcal E_{\sigma_0}(0)\right)^{\frac{p+1}2} 
\\
\mathcal E_\sigma(2\delta) & \le \mathcal E_\sigma(\delta) + C \delta  \sigma^2   \left( 2 \mathcal E_{\sigma_0}(0)\right)^{\frac{p+1}2} \le  \mathcal E_\sigma(0) + C 2\delta \sigma^2   \left( 2 \mathcal E_{\sigma_0}(0)\right)^{\frac{p+1}2} 
\\
& \cdots
\\
\mathcal E_\sigma(n\delta) & \le \mathcal E_\sigma(0) + C n \delta  \sigma^2   \left( 2 \mathcal E_{\sigma_0}(0)\right)^{\frac{p+1}2} .
\end{align*}
This continue as long as 
\begin{equation}\label{sigmabd-n}
 C n\delta  \sigma^2  \left( 2 \mathcal E_{\sigma_0}(0)\right)^{\frac{p+1}2}  \le  \mathcal E_{\sigma_0}(0) 
\end{equation}
since then 
\begin{equation*}
 \mathcal E_\sigma(n\delta) \le \mathcal E_{\sigma}(0) +  Cn\delta  \sigma^2   \left( 2 \mathcal E_{\sigma_0}(0)\right)^{\frac{p+1}2} \le 2 \mathcal E_{\sigma_0}(0)
  \end{equation*}
so we can take care one more step. Note also that \eqref{sigmabd-1} follows from \eqref{sigmabd-n}.

Thus, the induction stops at the first integer $n$ for which
$$
 Cn\delta  \sigma^2  \left( 2 \mathcal E_{\sigma_0}(0)\right)^{\frac{p+1}2}  >  \mathcal E_{\sigma_0}(0) 
$$
and then we have reached the final time 
$$
T=n\delta,
$$
when 
$$
 C T \sigma^2  \left( 2 \mathcal E_{\sigma_0}(0)\right)^{\frac{p-1}2}  >  1.
$$
Note that $T$ will be arbitrarily large for $\sigma >0$ small enough. Moreover, 
$$
\sigma^2 > C^{-1}\left( 2 \mathcal E_{\sigma_0}(0)\right)^{\frac{-p+1}2} \cdot  T^{-1} 
$$
 proving 
 $$
 \sigma \ge c T^{-\frac12},
 $$
as claimed.

\section{Proof of Lemma \ref{lm-error-est}}

First we prove the followings two Lemmas which are crucial in the proof of Lemma  \ref{lm-error-est}.
\begin{lemma} \label{lm-coshest0}
For $a, b \in \R$, we have
\begin{equation}\label{coshpest0}
\left| \cosh b - \cosh a   \right| \le \frac12 \Bigabs{ b^2-a^2} \left(\cosh b+  \cosh a\right).
\end{equation}
\end{lemma}
\begin{proof}
Note that $\cosh r$ is an increasing function for $r\ge 0$.
 Since $\cosh r$ is even, i.e., $\cosh r=\cosh |r|$,   we may assume $a, b\ge 0$. By symmetry we may also assume $b\ge a$. Then
\begin{align*}
\cosh b - \cosh a  = \int_a^b \int_0^s \cosh r \, dr \, ds \le   \cosh b \int_a^b \int_0^s \, dr \, ds= \frac 12 \left(b^2-a^2 \right)  \cosh b.
\end{align*}

\end{proof}

\begin{lemma} \label{lm-coshest}
Let $\xi=\sum_{j=1}^p \xi_j$ for $\xi_j \in \R^n$, where $p \ge 1$ is an integer. 
Then 
\begin{equation}\label{coshpest}
\left|1 -  \cosh |\xi| \prod_{j=1}^p  \sech |\xi_j|   \right| \le 2^p \sum_{  j\neq k =1}^p   |\xi_j| |\xi_k|.
\end{equation}

\end{lemma}
\begin{proof}
First observe that 
\begin{equation}\label{cosh-prod}
\prod_{j=1}^p  \cosh |\xi_j|= 2^{1-p} \sum_{s_2,  s_3 \cdots, s_p} \cosh \left( |\xi_1|+\sum_{j=2}^p s_j |\xi_j|\right),
\end{equation}
where $s_2, s_3, \cdots, s_p$ are independent signs ($+$ or $-$).
Indeed, the case $p=1$ is obvious, while the case $p=2$ follows from the identity
$$
2\cosh |\xi_1| \cosh |\xi_2| =   \cosh(|\xi_1|-|\xi_2|) +  \cosh(|\xi_1|+ |\xi_2|),
$$
which also implies \eqref{cosh-prod} for $p=3$.
The general case follows by induction.

It follows from \eqref{cosh-prod} that
\begin{equation}\label{cosh-prodest}
\cosh \left( |\xi_1|+\sum_{j=2}^p s_j |\xi_j|\right) + \cosh |\xi|  \le   2^{p}  \prod_{j=1}^p \cosh(|\xi_j|) .
\end{equation}
Observe also that 
\begin{equation}\label{Difest}
\left|  \left( |\xi_1|+\sum_{j=2}^p  s_j |\xi_j|\right)^2-  |\xi|^2 \right|  \le 2 \sum_{  
j\neq k =1}^p   |\xi_j| |\xi_k| .
\end{equation}

Applying \eqref{cosh-prod}, \eqref{coshpest0}, \eqref{cosh-prodest} and \eqref{Difest} we obtain
\begin{equation*}\label{cosh-prodest1}
\begin{split}
\left|  \prod_{j=1}^p \cosh |\xi_j|-\cosh |\xi|  \right| 
 & =    \left|  2^{1-p} \sum_{s_2,  s_3 \cdots, s_p} \cosh \left( |\xi_1|+\sum_{j=2}^p s_j |\xi_j|\right)-\cosh |\xi|  \right|
 \\
 &=  \left|  2^{1-p} \sum_{s_2,  s_3 \cdots, s_p} \left[  \cosh \left( |\xi_1|+\sum_{j=2}^p s_j |\xi_j|\right) -\cosh |\xi| \right] \right|
 \\
 &\le 2^{1-p} \sum_{s_2,  s_3 \cdots, s_p}  \frac12  \left|  \left( |\xi_1|+\sum_{j=2}^p  s_j |\xi_j|\right)^2-  |\xi|^2 \right| \left( \cosh \left( |\xi_1| +\sum_{j=2}^p s_j |\xi_j|\right) +  \cosh |\xi|  \right)
 \\
 &\le 2^{1-p} \sum_{s_2,  s_3 \cdots, s_p}  \left(\sum_{  j\neq k =1}^p   |\xi_j| |\xi_k| \right) \cdot   2^{p}  \prod_{j=1}^p \cosh(|\xi_j|) 
\\
 &=   2^{p}  \left(\sum_{  
j\neq k =1}^p   |\xi_j| |\xi_k| \right) \prod_{j=1}^p \cosh(|\xi_j|) .
\end{split}
\end{equation*}
Dividing by $ \prod_{j=1}^p \cosh(|\xi_j|) $
 yields the desired estimate \eqref{coshpest}.

\end{proof}

Now we prove Lemma \ref{lm-error-est}.
Recall that \begin{equation*}
N_p( v_\sigma)=   v_\sigma^p - \cosh (\sigma |D|) \left[  \sech (\sigma |D| )  v_\sigma  \right]^p.
\end{equation*}
By Cauchy-Schwarz inequality
$$
 \Bigabs{\int_{\R^n} \partial_t v_\sigma  \cdot  N_P\left( v_\sigma \right) \, dx }   \le 
  \norm{ \partial_t v_\sigma }_{L_x^2}    \norm{ N_P\left( v_\sigma \right) }_{L_x^2}   ,
$$
and so
we are reduced to prove
\begin{equation}\label{NpEst}
 \norm{ N_P\left( v_\sigma \right)  }_{L_x^2}     \lesssim \sigma^2    \norm{
v_\sigma  }^p_{ H^2}.
\end{equation}

Taking the Fourier transform we have
\begin{equation}
\label{NpFT}
 \begin{split}
\mathcal F_x [ N_p(v_\sigma) ](\xi) & = \int_{\xi= \sum_{j=1}^p \xi_j} \left[   1-  \cosh (\sigma |\xi|)  \prod_{j=1}^p  \sech (\sigma |\xi_j|)  \right]  \prod_{j=1}^p \widehat{v_\sigma}(\xi_{j})  \,   d\xi_1 d \xi_2 \cdots d \xi_p 
\end{split}
\end{equation}
By symmetry, we may assume $|\xi_1| \ge |\xi_2| \ge \cdots  \ge |\xi_p|$. 
By Lemma \ref{lm-coshest}, 
\begin{equation}\label{coshpest1}
\begin{split}
\left|    1-  \cosh (\sigma |\xi|)  \prod_{j=1}^p  \sech (\sigma |\xi_j|)  \right| & \le 2^p \sum_{  j\neq k =1}^p   |\sigma \xi_j| |\sigma \xi_k|
\\
 &  \le c(p) \sigma^2 |\xi_1| |\xi_2|,
\end{split}
\end{equation}
where $c(p)= p^2 2^p $.
Now set
$$ w_\sigma:= \mathcal F_x^{-1} \left( |\widehat{v_\sigma}|\right).$$

Then applying \eqref{coshpest1} to \eqref{NpFT} we obtain
\begin{align*}
|\mathcal F_x [ N_p(v_\sigma) ](\xi)  |& \leq c(p) \sigma^2 \int_{\xi= \sum_{j=1}^p \xi_j}  |\xi_1| |\xi_2| |\widehat{v_\sigma}(\xi_1)|  |\widehat{v_\sigma}(\xi_2)|     \prod_{j=3}^{p} |\widehat{v_\sigma}(\xi_{j}) |    \,  d\xi_1 d \xi_2 \cdots d \xi_p  
\\
&= c(p) \sigma^2 \int_{\xi= \sum_{j=1}^p \xi_j}     |\xi_1| |\xi_2| \widehat{w_\sigma}(\xi_1)\widehat{w_\sigma}(\xi_2)     \prod_{j=3}^{p} |\widehat{w_\sigma}(\xi_{j}) |   \, d\xi_1 d \xi_2 \cdots d \xi_p  
\\
&= c(p) \sigma^2  \mathcal F_x{\left( (|D|  w_\sigma )^2  \cdot w^{p-2}_\sigma \right)}(\xi).
\end{align*}
Therefore, using Plancherel, H\"{o}lder and Sobolev embedding we obtain
\begin{align*}
\| N_p(v_\sigma) \|_{L^{2}_x}  &\le  c (p)\sigma^2  \| (|D|  w_\sigma )^2  \cdot w^{p-2}_\sigma  \|_{L^{2}_x} 
\\
&\le c (p) \sigma^2 \| |D| w_\sigma\|^2_{L^{4}_x} \| w_\sigma \|_{L_x^\infty}^{p-2} 
\\
&\le C \sigma^2  \norm{w_\sigma}_{H^2}^2  \norm{w_\sigma}_{H^2}^{p-2}
\\
&= C \sigma^2 \norm{v_\sigma}_{H^2}^p
\end{align*}
as desired in \eqref{NpEst}.

\end{document}